\documentclass[12pt,a4paper]{article}
\usepackage{amsmath,amssymb,amsthm,hyperref}
\usepackage{graphicx}
\usepackage{tikz}
\usepackage{version} % include, exclude and comment
\definecolor{refkey}{rgb}{0.65,0.38,0.15}
\usepackage{enumerate}
\usepackage{datetime}
\numberwithin{equation}{section}
\usepackage[bf]{caption}     % captionstyles
\usepackage{geometry}        % for good layout
\usepackage{bm}  % bold symbols in math mode \boldsymbol{\alpha}
\usepackage[utf8]{inputenc}

\usepackage{buczobhmvdensity}
  \newtheorem{theorem}{Theorem}[section]

 \newtheorem{lemma}[theorem]{Lemma}

 {\theoremstyle{definition}
 \newtheorem{definition}[theorem]{Definition}

 }
\newcommand{\N}{\ensuremath{\mathbb N}} %natural numbers
\newcommand{\R}{\ensuremath{\mathbb R}} %real numbers
\newcommand{\beq}{\begin{equation}}
\newcommand{\eeq}{\end{equation}}
\newcommand{\be}{\begin{enumerate}}
\newcommand{\ee}{\end{enumerate}}
\newcommand{\bi}{\begin{itemize}}
\newcommand{\ei}{\end{itemize}}
\newcommand{\rr}{\mathbb R}

\DeclareMathOperator{\lip}{lip\kern-0.8pt}
\DeclareMathOperator{\Lip}{Lip\kern-0.8pt}

\DeclareMathOperator{\dt}{\mathrm{d}t}

\definecolor{skyblue}{rgb}{0,0.4,0.6}
\definecolor{red}{rgb}{0.6,0,0}
\definecolor{green}{rgb}{0,0.6,0}
\definecolor{aquam}{rgb}{0.5,1.0,1.0}

\definecolor{bbrown}{rgb}{0.75,0.38,0.15}
\definecolor{Cyan}{rgb}{0,0.6,0.6}
\definecolor{Darkblue}{rgb}{0,0,0.9}
\definecolor{Dodgerblue2}{rgb}{0,0.5,1}
\definecolor{Green}{rgb}{0,0.5,0.1}
\definecolor{dGreen}{rgb}{0,0.5,0}
\definecolor{Kahki}{rgb}{1,1,0.5}
\definecolor{Magenta}{rgb}{1,0,1}
\definecolor{bMagenta}{rgb}{1,.6,1}
\definecolor{negyedik}{rgb}{0.9,0.6,0}
\definecolor{Orange}{rgb}{0.8,0.3,0}
\definecolor{dOrchid}{rgb}{0.7,0.2,0.4}
\definecolor{Orchid}{rgb}{1,0.5,1}
\definecolor{Purple}{rgb}{0.65,0.07,0.85}
\definecolor{Royalblue}{rgb}{0.6,0.85,0.87}
\definecolor{Tan}{rgb}{0.54,0.42,0.23}
\definecolor{bTan}{rgb}{0.94,0.82,0.63}
\definecolor{Turquoise}{rgb}{0,0.85,0.87}
\definecolor{Yellow}{rgb}{1,1,0}
\definecolor{bYellow}{rgb}{1,1,0.6}
\definecolor{bRed}{rgb}{1,0.7,0.7}
\definecolor{dRed}{rgb}{0.7,0,0}
\definecolor{dRed}{rgb}{1,0,0}
\definecolor{boxcolb}{rgb}{0.87,0.77,0.75}%rosybrown
\definecolor{boxcol}{rgb}{0.6,0.85,0.87}%cadetblue
\definecolor{boxcolgreen}{rgb}{0.64,0.93,0.79}
\definecolor{boxcolaa}{rgb}{.75,.99,.70}
\definecolor{boxcolbb}{rgb}{0.39,0.50,0.56}
\definecolor{boxcolcc}{rgb}{1,0.81,0.65}
\definecolor{yy}{rgb}{0.43,0.21,.18}
\definecolor{gA}{gray}{0.5}
\definecolor{gB}{gray}{0.8}
\definecolor{gC}{gray}{0.9}
\title{ Strong one-sided density without uniform density}
\author{Zolt\'an Buczolich\thanks{\scriptsize
This author was supported by the Hungarian National Research, Development and Innovation Office--NKFIH, Grant 124003.
},
Department of Analysis, ELTE E\"otv\"os Lor\'and\\
University, P\'azm\'any P\'eter S\'et\'any 1/c, 1117 Budapest, Hungary\\
email: \texttt{zoltan.buczolich@ttk.elte.hu}\\
{\tt  http://buczo.web.elte.hu}\\
ORCID Id: 0000-0001-5481-8797
   \smallskip\\
  Bruce Hanson, Department of Mathematics,\\ Statistics and Computer Science,\\ St.\ Olaf College,
Northfield, Minnesota 55057, USA\\
{email:} \texttt{hansonb@stolaf.edu}
  \smallskip\\
 Bal\'azs Maga\thanks{\scriptsize This author was supported by the \'UNKP-20-3 New National Excellence Program of the Ministry for Innovation and Technology from the source of the National Research, Development and Innovation Fund, and by the Hungarian National Research, Development and Innovation Office–NKFIH, Grant 124749.},
Department of Analysis, ELTE E\"otv\"os Lor\'and\\
University, P\'azm\'any P\'eter S\'et\'any 1/c, 1117 Budapest, Hungary\\
 email: magab@caesar.elte.hu \\{\tt  http://magab.web.elte.hu/}
   \smallskip\\
and
   \smallskip\\
 G\'asp\'ar V\'ertesy\thanks{\scriptsize This author was supported by the \'UNKP-20-3 New National Excellence Program of the Ministry for Innovation and Technology from the source of the National Research, Development and Innovation Fund, and by the Hungarian National Research, Development and Innovation Office–NKFIH, Grant 124749.
},
 Department of Analysis, ELTE E\"otv\"os Lor\'and\\
University, P\'azm\'any P\'eter S\'et\'any 1/c, 1117 Budapest, Hungary\\
email: vertesy.gaspar@gmail.com\
}
\date{\today}
 
\begin{document}
\maketitle

\newpage

 \begin{abstract}
 In this paper we give an example of a closed, strongly one-sided dense set which is not of uniform density type. We also  show that there is  a set of uniform density type which is not of strong uniform density type.
   \end{abstract}

 \section{Introduction}\label{*secintro}

 The  ``big Lip'' and ``little lip'' functions are defined as follows:

 \begin{equation}
  \Lip
  f(x)= \limsup_{r\to 0^+}M_f(x,r),\qquad\label{Lipdef}
  \lip
  f(x)= \liminf_{r\rightarrow 0^+}M_f(x,r),
 \end{equation}
where
$$M_f(x,r)=\frac{\sup\{|f(x)-f(y)| \colon |x-y| \le r\}}r.$$

While the definition of the $\Lip$ function has a long history, the definition of
$\lip f$ is more recent see \cite{[Cheeger]} and  \cite{[Keith]}.
For some more recent appearances of $\lip f$ we refer to \cite{BaloghCsornyei}, \cite{Hanson}, \cite{Hanson2},  \cite{Hanson3},\cite{BHRZ}, 
\cite{MaZi}, and \cite{[Zi]}. 

In our recent research project (see \cite{[BHMVlip]}, \cite{[BHMVlipchar]}, and   \cite{[BHMVlipap]}) we were primarily concerned with the characterization of $\Lip 1$ and $\lip 1$ sets. Notably, a set $E \subset \rr$ is a $\Lip 1$ ($\lip 1$) set if there exists a continuous function $f$ defined on $\rr$ such that $\Lip f=\mathbf{1}_E$ ($\lip f = \mathbf{1}_E$.) In the aforementioned papers we demonstrated that these properties are intimately related to certain density notions introduced in \cite{[BHMVlip]}, such as weak density (WD),  strong  one-sided density (SOSD), uniform density type (UDT), and strong uniform density type (SUDT).  (In Section \ref{prelim*} the reader can find the definition of UDT sets as well as the definitions of several other less well-known concepts appearing in this introduction.)   
 In \cite{[BHMVlipchar]} the $\lip 1$ sets were characterized  as countable unions of closed sets which are strongly one-sided dense. 
 Such a nice characterization of $\Lip 1$ sets  has not yet been discovered.  
  The  main result of \cite{[BHMVlip]} states   that if $E$ is $G_\delta$ and $E$ has UDT then there exists a continuous function $f$ satisfying $ \Lip f =\mathbf{1}_E$, that is, the set $E$ is $ \Lip 1$. 
 On the other hand, in \cite{[BHMVlipap]} we showed that the converse of this statement does not hold.  There exist $\Lip 1$ sets which are not UDT.
 Moreover, in  \cite{[BHMVlip]} we showed that if $E\subset {\ensuremath {\mathbb R}}$ is $ \Lip 1$ then $E$ is a weakly dense $G_\delta$  set.    We also showed that there exists a weakly dense $G_\delta$ set $E\subset {\ensuremath {\mathbb R}}$ which is not $ \Lip 1$, thus this condition on $E$ is necessary, but not sufficient. 
 We can summarize some of the above results as
\begin{equation} \label{Lip_implication_chain}
 G_{\delta}\& \text{UDT}\ \substack{\Rightarrow \\ \not \Leftarrow}\ \Lip 1\ \substack{\Rightarrow \\ \not \Leftarrow}\  G_{\delta}\& \text{WD}.
\end{equation}

 The connection between  these  density  properties  and $\Lip 1$ and $\lip 1$ sets established by the abovementioned results  provide the motivation for examining the intrinsic relationships between these properties.    For instance, in \cite{[BHMVlip]} we showed that UDT sets are strongly one-sided dense. Coupled with trivial observations, we  obtained  the following chain of implications:
\begin{equation} \label{eq:density_implication_chain}
 \text{SUDT} \Rightarrow \text{UDT} \Rightarrow \text{SOSD} \Rightarrow \text{WD}.
\end{equation}

Partially inspired by the  comments of the referee of \cite{[BHMVlip]},  we were interested in  determining    whether any  of the  implications in 
\eqref{eq:density_implication_chain} 
  are   reversible. Concerning the last one,  showing that WD does not imply SOSD is very easy.  Take for example the set 
 \begin{equation}\label{WD_not_SOSD}
 E=\{ 0 \}\cup \bigcup_{n=2}^{\oo} [n^{-n-0.5},n^{-n}],
 \end{equation}  
 which   is WD, but not SOSD at $0$.
However, the other cases turned out to be more difficult. 
The main results of our paper  provide answers to   these questions: in Theorem \ref{thm:main} 
 we show that the SUDT property is  strictly stronger  than the UDT property by giving an example of a UDT set which is not SUDT. Moreover, 
 in Theorem \ref{*thmudtnosudt} 
we provide an example of a closed, strongly one-sided dense set which is not UDT.

So in this paper we prove:
\begin{equation} \label{*revnoimplication_chain}
 \text{SUDT} \not\Leftarrow \text{UDT} \not\Leftarrow \text{SOSD} \not\Leftarrow \text{WD}.
\end{equation}

 Finally, a result with a somewhat different flavor is provided  concerning   UDT and SUDT sets:  we prove that any measurable set equals an SUDT set almost everywhere.

% \marginpar{\tiny I'm wondering whether we should omit this paragraph, it's somewhat loosely connected to the rest of the introduction.  see my modification.}  
 In measure theory people  have gotten   used to the fact that usually a ``less nice" set coincides with a ``nice" set modulo a set of measure zero, and quite often, but not always, it is not too difficult to prove it. In our research project concerning 
$\lip 1$ and $\Lip 1$ sets we have already looked at this type of questions.  
In \cite{[BHMVlipap]}
 we proved that there exists a measurable SUDT set $E$  such that for any $G_\delta$ set $\widetilde{E}$ satisfying $|E\Delta\widetilde{E}|=0$ the set $\widetilde{E}$ does not have UDT. 
 In the same paper we also showed that modulo sets of zero Lebesgue measure any measurable set coincides with a $\Lip 1$ set  (this was a case when it was not that easy to prove that ``less nice" can be approximated by ``nice" modulo a set of measure zero).   
 In fact, these two results imply that there exist $\Lip 1$ sets not having UDT, a result we mentioned earlier. 
 However, \eqref{WD_not_SOSD} provides a much simpler example of such a set. It is easy to see that for $f(x):=\int_0^x \mathbf{1}_E(t) \dt$ we have $\Lip f = E$ (or we can apply Theorem 3.1 from \cite{[BHMVlip]} to $E$).

\section{ Preliminary definitions  }\label{prelim*}
 
 In this section, we define the four density notions our paper  is concerned   with.  We denote the closure of the set $A$ by $\overline{A}$, its complement by $A^{c}$,  and its Lebesgue measure by $|A|$, if it exists. 

 \begin{definition} The set $E$ is {\it strongly one-sided dense}   (SOSD)  at $x$ if for any sequence %$\{I_n\}=\{[x-r_n,x+r_n]\}$ such that
 $r_n \to 0^+$  we have $$\max\Big \{\frac{|E \cap [x-r_n,x]|}{r_n},\frac{|E \cap [x,x+r_n]|}{r_n}\Big \}\to 1.$$
 % In other words, $x$ is a strongly one-sided density point of $E$.
The set $E$ is {\it strongly one-sided dense} if $E$ is strongly one-sided dense at every point $x \in E$.
\end{definition}

  \begin{definition}\label{*defudt}
Suppose that $E\subseteq\mathbb{R}$ is measurable and $\gamma,\delta>0$.
Let
  \begin{displaymath} E^{\gamma,\delta}=\left\{x\in\mathbb{R}:
 \forall r\in (0,\delta],\texttt{}
 \max\left\{\frac{|(x-r,x)\cap E|}{r},\frac{|(x,x+r)\cap E|}{r}\right\}\geq\gamma \right\}.
\end{displaymath}

We say that $E$ has uniform density type (UDT) if there exist sequences $\gamma_n\nearrow 1$ and $\delta_n\searrow 0$ such that $E\subseteq  \bigcap_{k=1}^{\infty}  \bigcup_{n=k}^{\infty}E^{\gamma_n,\delta_n}$.

We say that $E$ has strong uniform density type (SUDT) if there exist sequences $\gamma_n\nearrow 1$ and $\delta_n\searrow 0$ such that $E\subseteq  \bigcup_{k=1}^{\infty}  \bigcap_{n=k}^{\infty}E^{\gamma_n,\delta_n}$.
\end{definition}

%  The motivation for the present paper is an earlier paper of ours \cite{[BHMVlip]},
%which concerns the 
%  characterization of $ \Lip 1$ and $ \lip 1$ sets. These are 
%sets $E$ in $ {\mathbb R}$  for which there is a continuous function defined on $ {\mathbb R}$ such that $ \Lip f =\mathbf{1}_E$, (or $ \lip f=\mathbf{1}_E$, respectively). 
 
  \begin{definition}\label{weak dense}
Given a sequence of non-degenerate closed intervals $\{I_n\}$, we write {\em $I_n\to x$} if $x \in I_n$ for all
$n \in  \mathbb{N}$ and $|I_n|\to 0$.

The measurable set
$E$ is {\it weakly dense}  (WD)  at $x$ if there exists $I_n \to x$ such that $\frac{|E \cap I_n|}{|I_n|}\to 1$.
The set $E$ is {\it weakly dense} if $E$ is weakly dense at $x$ for each $x \in E$.

\end{definition}

 \section{ SOSD $\not\Rightarrow $ UDT }\label{maina*}

 The following theorem establishes that the second implication in \eqref{eq:density_implication_chain} cannot be reversed.

 \begin{theorem} \label{thm:main} There exists a closed, strongly one-sided dense set which does not have UDT. \end{theorem}

 \begin{proof}
We will define a closed strongly one-sided dense set $E$ with the following property: 
for any sequences $\gamma_n\nearrow 1$ and $\delta_n\searrow 0$, we can find $ x^*=x_{(\gamma_n, \delta_n)_{n\in\mathbb{N}}}\in E$ such that 
$ x^* \notin  \bigcap_{k=1}^{\infty}  \bigcup_{n=k}^{\infty}E^{\gamma_n,\delta_n}$.
\begin{comment}
, nevertheless $E$ is strongly one-sided dense at $x$. 
This would be sufficient: by Lebesgue's density theorem, 
almost every point of $E$ is a Lebesgue density point and hence a strongly one-sided density point.
 As these points remain strongly one-sided density points after the removal of the null set of exceptional points, we can find $E_{*}\subseteq E$ of the same measure which is strongly one-sided dense and contains all strongly one-sided density points of $E$ which are of course 
 strongly one-sided density points of $E_{*}$ as well.
 However, $E^{\gamma_n,\delta_n} = E_{*}^{\gamma_n,\delta_n}$. 
 \end{comment}
Hence $E$ cannot be UDT: the points $x_{(\gamma_n, \delta_n)_{n\in\mathbb{N}}}$ demonstrate that no sequence $\gamma_n, \delta_n$ can be a valid choice in the definition.

In order to define this set $E$, we will recursively define sequences. 
First, we define $a_n = 2^{-n}$. 
Proceeding  by recursion  with respect to  $k$, if $a_{n_1,\ldots, n_k}\in\mathbb{R}$ is already defined for $k$ and all $n_1,\ldots,n_k \in \mathbb{N}$, then we define

\begin{equation}\label{seq def}
a_{n_1,n_2,...,n_k,n} = a_{n_1,n_2,...,n_{k-1},n_k+1}+\frac1{2^{n+3}} r_{n_1, ..., n_k},
\end{equation}
where
$$r_{n_1, ..., n_k} = a_{n_1, ..., n_k} - a_{n_1, ...,n_{k-1}, n_k+1}.$$

To get familiar with this definition, and for later use we calculate
\begin{equation}\label{*calcdef}
a_{1}=2^{-1},\ a_{2}=2^{-2},\ a_{11}=a_{2}+\frac{1}{2^{4}}(a_{1}-a_{2})=2^{-2}+2^{-2}\cdot 2^{-4},
\end{equation} 
$$a_{12}=a_{2}+\frac{1}{2^{5}}(a_{1}-a_{2})=2^{-2}+2^{-2}\cdot 2^{-5},\ 
r_{11}:=a_{11}-a_{12}=2^{-2}\cdot 2^{-5}.$$ 
%$$a_{111}=a_{12}+\frac{1}{2^{4}}(a_{11}-a_{12})=a_{12}+{2^{-4}}\cdot 2^{-2}\cdot 2^{-5}.$$
 \begin{figure}[!hb]
 \begin{center}
\includegraphics[width=1 \textwidth] {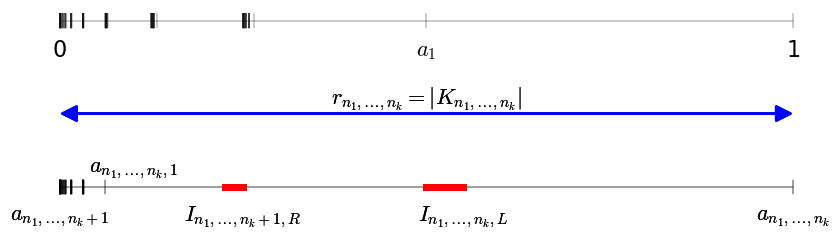}
%\vspace*{-1cm}
\caption{The set $A$  and its part with some objects used in the proof}\label{fig_egesz}
\end{center}
\end{figure}

% \begin{figure}[!ht]
 %\begin{center}
%\includegraphics[width=1 \textwidth] {lip_density_resz.pdf}
%\vspace*{-1cm}
%\caption{The set $A$}\label{fig_egesz}
%\end{center}
%\end{figure}
% \text{ $a_{n_1, ..., n_k, n} \searrow a_{n_1, ..., n_k+1}$}

 Note that  for the decreasing sequence $(a_{n_1, ..., n_k, n})_{n=1}^{\infty}$
 \begin{equation}\label{*mondec}
 \lim_{n\to\infty} a_{n_1, ..., n_k, n} = a_{n_1, ..., n_k+1},
 \end{equation}
  and the following conditions are satisfied:
 \begin{enumerate}[(i)]
\item  $a_{n_1, ...,n_{k-1}, n_k+1} < a_{n_1, ..., n_k, n} < a_{n_1, ..., n_k}$  for any $k$, any indices $n_1, ..., n_k$, and $n \in \mathbb{N}$,
\item \label{uniformity} 
$$ a_{n_1, ..., n_{k-1}, n_k+1} - a_{n_1, ...,n_{k-1}, n_k+2} = \frac{a_{n_1, ...,n_{k-1}, n_k} - a_{n_1, ...,n_{k-1}, n_k+1}}{2}$$
for any $k$ and any indices $n_1, ..., n_k$.
%\item \label{departed} for any $k < k'$ and any indices $n_1, ..., n_k, ... n_{k'}$, if $n_k>1$, then we have
%$$a_{n_1, ..., n_k, ... n_{k'}} \le a_{n_1, ..., n_k}+ \frac{a_{n_1, ..., n_k-1}-a_{n_1, ..., n_k}}{6}.$$
%(We note that (iii) implies that for any $k$ and any indices $n_1, ..., n_k$, then $a_{n_1, ..., n_k, n} \in (a_{n_1, ..., n_k+1}, a_{n_1, ..., n_k})$.)
\end{enumerate}

%These properties can be guaranteed easily as the recursion proceeds: as the sequences are decreasing with the above prescribed limits, (ii) depend only on the starting terms of each sequence, and if (ii) holds for a certain $a_{n_1, ..., n_k}$ and $a_{n_1, ..., n_k, ... n_{k'}}$, we can clearly choose $a_{n_1, ..., n_k, ... n_{k'}, 1}$ accordingly. 
%Finally, if the starting terms are chosen according to (i), property (ii) determines the remaining terms. 

For each $k\in  {\ensuremath {\mathbb N}}$ define $A_k=\{a_{n_1,n_2,\ldots,n_k} \,:\, n_1,n_2,\ldots,n_k \in  {\ensuremath {\mathbb N}}\}$ and let $A=\bigcup_{k=1}^\infty A_k$. 
 Observe that $A \subset (0,\frac12]$. 
The conditions above yield that the set $A$ 
is disjoint from all the sets
$$ \left[a_{n_1, ..., n_k} - \frac12r_{n_1, ..., n_k}, a_{n_1, ..., n_k}-\frac14r_{n_1,...,n_k}\right] \cup \left[a_{n_1, ..., n_k} + \frac{1}{4}r_{n_1, ..., n_k}, a_{n_1, ..., n_k} + \frac12r_{n_1, ..., n_k} \right].$$

Set $\alpha_k=10^{-k}$, $\gamma_k=1-\alpha_k$ and for $n_1,n_2,\ldots,n_k \in  {\ensuremath {\mathbb N}}$ define 
$$I_{n_1,n_2,...,n_k}^L=\left(a_{n_1, ..., n_k} - \frac12r_{n_1, ..., n_k}, a_{n_1, ..., n_k}  -\Big(\frac12 - \alpha_k\Big)  r_{n_1,...n_k}\right),$$
$$I_{n_1,n_2,...,n_k}^R=\left(a_{n_1, ..., n_k} + \Big(\frac12-\alpha_k\Big)r_{n_1, ..., n_k}, a_{n_1, ..., n_k}+\frac12r_{n_1,...n_k}\right),$$
and 
$$I_{n_1,n_2,...,n_k}=I_{n_1,n_2,...,n_k}^L\cup I_{n_1,n_2,...,n_k}^R.$$

Observe that
\begin{equation}\label{disjoint}
\begin{gathered}
\{\overline{I}_{n_1,n_2,...,n_k}^L : k,n_1,n_2,\ldots,n_k \in\N\} \cup \{\overline{I}_{n_1,n_2,...,n_k}^R : k,n_1,n_2,\ldots,n_k \in\N\}\\
\text{consists of pairwise disjoint closed intervals.}
\end{gathered}
\end{equation}
We also define 
$$J_{n_1,n_2,...,n_k}^L=\left[a_{n_1, ..., n_k} - \frac12r_{n_1, ..., n_k}, a_{n_1, ..., n_k}\right]$$
and 
$$J_{n_1,n_2,...,n_k}^R=\left[a_{n_1, ..., n_k}, a_{n_1, ..., n_k}+\frac12r_{n_1,...n_k}\right].$$
Then set 
$$E=[-1,1]\backslash  \bigcup_{n_1,...,n_k\in  {\ensuremath {\mathbb N}}} I_{n_1,...,n_k}.$$
%We claim that $E$ is strongly one-sided dense but not UDT.  

We first demonstrate that $E$ is not UDT.   Suppose otherwise, so there exist $\widetilde{\gamma}_n \nearrow 1$ and 
$\widetilde{\delta}_n \searrow 0$ such that 
\begin{equation}\label{E subset}
E\subset  \bigcap_{k=1}^{\infty}  \bigcup_{n=k}^{\infty}E^{\widetilde{\gamma}_n,\widetilde\delta_n}.
 \end{equation}

For each $n\in  {\ensuremath {\mathbb N}}$ let 
$$\delta_n=\min\{\widetilde\delta_{n'} \,:\, \widetilde\gamma_{n'} < \gamma_{n+1} =1-10^{-(n+1)} \}.$$
 %\marginpar{\tiny  In accordance with the ref's comment, I've put this reminder here. I find it rather unnecessary though to recall this numerical value, as it lacks importance for quite a while, so why bother here? One has to look it up much later, it does not provide support.
%  Let's make the referee happy.} 
Then for all $n'$ satisfying $\gamma_n \le \widetilde\gamma_{n'} < \gamma_{n+1}$ we have 
$E^{\widetilde\gamma_{n'},\widetilde\delta_{n'}}\subset E^{\gamma_n,\delta_n}$ and therefore from \eqref{E subset} we get

 \begin{equation}\label{E subset2}
E\subset  \bigcap_{k=1}^{\infty}  \bigcup_{n=k}^{\infty}E^{{\gamma}_n,\delta_n}.
 \end{equation}

From the definitions it follows that
if $J=J_{n_1,...,n_k}^L$ or $J=J_{n_1,...,n_k}^R$, where $n_k \ge  1 $, we have 
$$ \frac{|E \cap J|}{|J|}  \le 1-2\alpha_k< \gamma_k.$$
It follows from this that for each $x =a_{n_1,...,n_k} \in A$ we have 
 \begin{equation}\label{gamma k ineq}
\begin{gathered}
\max\left\{\frac{|(x-r,x)\cap E|}{r},\frac{|(x,x+r)\cap E|}{r}\right\}<\gamma_k, \\
\text{    for    } r=|J_{n_1,...,n_k}^L|=|J_{n_1,...,n_k}^R|=\frac12 r_{n_1,...,n_k}.
\end{gathered}
\end{equation}
 Then, by continuity inequality \eqref{gamma k ineq} holds in a neighborhood of $x$ as well.   

We now describe how to choose a point $x^*\in E$ such that $x^* \notin  \bigcap_{k=1}^{\infty}  \bigcup_{n=k}^{\infty}E^{{\gamma}_n,\delta_n},$ which will give the desired contradiction to  \eqref{E subset2}.   To streamline the notation we use the following convention:
given $x=a_{n_1,...,n_k} \in A$, we define $r_x=\frac12r_{n_1,...,n_k}=|J_{n_1,...,n_k}^L|=|J_{n_1,...,n_k}^R|.$ 
If $n_k$ is sufficiently large, then $r_x < \delta_k$ and 
hence from \eqref{gamma k ineq} we know that there is a  neighborhood  $U_x$ of $x$ such that
%Then from the previous paragraph we know that for each $x \in A_k$ there is a neighborhood $U_x$ of $x$ such that 
 \begin{equation}\label{x not in E^}
U_x \cap  E^{\gamma_k,\delta_k}=\emptyset.
 \end{equation}

%The following fact which can be easily verified from the construction of $A$ will be useful in demonstrating the existence of $x^*$. 

%{\bf Fact:} Given any $x \in A_k$, $\delta >0$ and open interval $I$ containing $x$, we can find an open subinterval $I'$ of $I$ such that $x\in I'$ and for every $x' \in A_{k+1}\cap I'$ such that  $r_{x'} <\delta$. 
%Given any $x \in A_k$, $\delta >0$ and open interval $I$ containing $x$, we can find $x' \in A_{k+1}\cap I$ such that $r_{x'} <\delta$. 

Thus, we can choose $n_1 \in \N$ and a closed interval $I_1$ containing $a_{n_1}$ in its interior  such that $n_1>1$ and $I_1 \cap E^{\gamma_1,\delta_1}=\emptyset.$  Proceeding inductively and using \eqref{x not in E^}, we choose a sequence of 
integers $\{n_k\}$ and closed intervals $\{I_k\}$ such that for all $k \in  {\ensuremath {\mathbb N}}$ we have
 \begin{itemize} 
 
 \item $n_k>1,$
 
 \item  $a_{n_1-1,n_2-1,\ldots,n_{k-1}-1,n_k}\in \text{    int    }(I_k),$  (keeping in mind \eqref{*mondec}) 
 
 \item $I_k \supset I_{k+1},$

\item $I_k \cap E^{\gamma_k,\delta_k} = \emptyset$.

\end{itemize} 

Since $E$ is closed, it follows that $E \cap (\bigcap_{k=1}^\infty I_k )\neq \emptyset$ and for any
$x^* \in  E \cap (\bigcap_{k=1}^\infty I_k) $ we have that $x^* \notin  \bigcap_{k=1}^{\infty}  \bigcup_{n=k}^{\infty}E^{{\gamma}_n,\delta_n}.$   This 
 concludes  the argument that $E$ is not UDT. 

Next we show that all points  $x\in E$ are strong one-sided density points of $E$.

%\begin{comment}
 
First suppose that $x \notin \overline{A}$. 
In this case we claim that 
\begin{equation}\label{intervallumbeli}
\text{there is a closed interval $I$ such that $x \in I \subset E$ }
\end{equation}
and therefore $E$ is strongly one-sided dense at $x$. 
To see the truth of our claim let 
\begin{equation}\label{J def}
\begin{gathered}
\mathcal{J}=\bigcup_{k=1}^\infty \mathcal{J}^k, \text{    where    } \\
\mathcal{J}^k=\Big \{[ {a_ {n_1,...,n_k}}-\frac12 r_{n_1,...,n_k},{a_ {n_1,...,n_k}}+\frac12r_{n_1,...,n_k}] \,:\, n_1,...,n_k \in  {\ensuremath {\mathbb N}}\Big \}.
\end{gathered}
 \end{equation}
For each $J=[ {a_ {n_1,...,n_k}}-\frac12 r_{n_1,...,n_k},{a_ {n_1,...,n_k}}+\frac12r_{n_1,...,n_k}] \in \mathcal{J}$ we define $J'=I_{{n_1,...,n_k}}$. 
For $ \epsilon >0$, define $\mathcal{J}_ \epsilon=\{J \in \mathcal{J} \,:\, |J| \ge  \epsilon\}$  and let 
$$E_ \epsilon =[-1,1] \backslash \bigcup_{J \in \mathcal{J}_ \epsilon} J'.$$
  Then each $E_ \epsilon$ is a finite union of closed intervals and $E = \bigcap_{\epsilon>0} E_ \epsilon.$  Choose $ \epsilon = \frac12 \text{    dist    }(\{x\},\overline{A}).$  
  Thus, we have $x \in I' \subset E_ \epsilon$ for some closed interval $I'$. It may happen that $x$ is an endpoint of $I'$, but  $\ds I'$ is non-degenerate by \eqref{disjoint}.
Observe that $E$ is constructed from $E_ \epsilon$ by removing unions of open intervals $J'\subset J$, where $J$ is centered at  a 
point of $A$ and $|J| < \epsilon$. 
Since $ \epsilon \leq \frac12 \text{    dist    }(\{x\},\overline{A})$, it follows that the distance from $x$ to any of the intervals being removed is at least $\frac12 \text{    dist    }(\{x\},A)$ and therefore there must be a
 non-degenerate closed subinterval $I$ of $I'$ such that $x \in I \subset E$.   
%\end{comment}

%If $x \notin \overline{A}$, then, by \eqref{disjoint}, there is a non-degenerate closed interval $I$ such that $x \in I \subset E$ and therefore $E$ is strongly one-sided dense at $x$. 

 Now suppose that $x \in \overline{A}$  and recall that  $\alpha_k=10^{-k}$. 
 %\marginpar{\tiny  I'm unsure about recalling this here, see my similar comment above. 
 % Let's make the referee happy.}
 
To finish the proof we need the following technical lemma:

 \begin{lemma}\label{interval density}
For any $ {n_1,\ldots,n_{k}} \in  {\ensuremath {\mathbb N}}$, we let $K_{{n_1,...,n_{k}}}=[a_{{n_1,...,n_{k-1},n_k+1}}, {a_ {n_1,..., n_{k-1}, n_k}}]$. 
Then 
 \begin{equation}\label{K density}
|E \cap K_{{n_1,...,n_{k}}}| \ge (1-2\alpha_k) |K_{{n_1,...,n_{k}}}| \text{    for every    }  {n_1,...,n_{k}} \in  {\ensuremath {\mathbb N}}.
 \end{equation}
\end{lemma}

 \begin{proof}

For $n\in  {\ensuremath {\mathbb N}}$ we set 
$$E^n=[-1,1]\backslash  \bigcup_{\substack{n_1,...,n_k\in  {\ensuremath {\mathbb N}} \\ k \le n}} I_{n_1,...,n_k}.$$
Note that $E = \bigcap_{n=1}^\infty E^n$.

Let $K=K_{{n_1,...,n_{k}}}=[a_{{n_1,...,n_{k-1},n_k+1}}, {a_ {n_1,..., n_{k-1}, n_k}}]$. 
Consider the set $E^k \cap K$.  
This set  consists of  the interval $K$ with two open intervals $I_{{n_1,...,n_k}}^L$ and $I_{{n_1,...,n_k}+1}^R$ of total length $\frac32 \alpha_k |K|$ removed. 
At the next stage of the construction, to create the set $E^{k+1}\cap K$, for each $J \in \mathcal{J}^{k+1}$ such that $J \subset K$ we remove two open intervals of total length equal to 
$2 \alpha_{k+1} |J|$. 
Since the intervals in $\mathcal{J}^{k+1}$ are non-overlapping, it follows that the total length of the intervals removed from $K$ at the $(k+1)$st stage of the construction is less than
$2\alpha_{k+1} |K|$. 
An entirely analogous argument shows that at the $j$th stage of the construction where $j >k$, the total length of the intervals removed from $K$ is no more than
$2 \alpha_j |K|$. 
It follows that 
$$|E\cap K| \ge \Big (1 - \frac32\alpha_k - \sum_{n=k+1}^\infty 2\alpha_n\Big ) |K|> (1-2\alpha_k)|K|,$$ as desired.
\end{proof}

 Note, first of all, that if $x \in A$, then from the construction of $E$ we see that there exists a closed interval $I=[x-\delta,x]$ such that 
$I \subset E$ and therefore $E$ is strongly one-sided dense at $x$. 

So we may assume that $x \notin A$. 
Also, note that $0 \in [-1,0]\subset E$ so $E$ is strongly one-sided dense at $0$ and therefore we may assume that $x>0$.   

Note that since $x \in \overline{A}\backslash A$, we can choose a sequence of indices $n_1,n_2, \ldots$ and a nested sequence of closed intervals 
$K'_1\supset K'_2 \supset K'_3 \supset \ldots$ such that for each $k \in  {\ensuremath {\mathbb N}}$ we have 
$x \in K'_k= K_{{n_1,...,n_{k}}}$.
For each $k \in  {\ensuremath {\mathbb N}}$ we let $r'_k = |K'_k|$ so that $r'_k \searrow 0$.   
 Observe  that 
\begin{equation}\label{*capx}
\{x\}=\bigcap_{k=1}^\infty K'_k. 
\end{equation}

Assume that $r'_{k+1} \le r \le r'_{k}$. 
We consider two cases: 

\noindent
{\bf Case 1:} Assume that $\frac1{64} r'_{k} \le r \le r'_{k}$.\\
In this case we first assume that $n_k > 1$. 
Let $K=K_{{n_1,...,n_{k}}}=[a_{{n_1,...,n_{k-1},n_{k}+1}},a_{n_1,...,n_k}]$ and $K_*=K_{n_1,...,n_{k-1},n_k-1}=[a_{{n_1,...,n_k}},a_{{n_1,...,n_{k-1},n_k-1}}]$. 
Then $[x,x+r] \subset K \cup K_*$ and it follows from Lemma \ref{interval density} that 
$$|E^c \cap [x,x+r]| \le 2\alpha_k (|K|+|K_*|)= 6 \alpha_k r'_{k} \le 384 \alpha_k r,$$
so we get 
 \begin{equation}\label{case 1}
|E \cap [x,x+r]| \ge (1-384 \alpha_k) r.
 \end{equation}
 
A similar argument gives the same estimate if $n_k=1$. 
We leave the details up to the reader.

\noindent
{\bf Case 2:}   Assume that $r'_{k+1} \le r \le \frac1{64} r'_{k}$.   

 Note that by \eqref{seq def}, we have $|K_{{n_1,...,n_k,1}}| =\frac1{32}r'_k > r'_{k+1}$ and therefore $n_{k+1}\ge 2$.  
%\begin{align*}
%a_{{n_1,...,n_k},1,1} &= a_{{n_1,...,n_k},2}+\frac1{16}(a_{{n_1,...,n_k},1}-a_{{n_1,...,n_k},2}) \\
%&= a_{{n_1,...,n_k}+1} + \frac1{32}r'_k + \frac1{16}\Big(a_{{n_1,...,n_k}+1}+\frac1{16}r'_k-a_{{n_1,...,n_k}+1}-\frac1{32}r'_k\Big) \\
%&= a_{{n_1,...,n_k}+1} + \frac1{32}r'_k + \frac1{512}r'_k.
%\end{align*}

It follows that  $x \in  K'_{k+1}=K_{{n_1,...,n_{k},n_{k+1}}} \sse  [a_{{n_1,...,n_{k-1},n_{k}+1}},a_{{n_1,...,n_{k}},2}]$, 
and we obtain 
$$
[x,x+r]\subset \Big[a_{{n_1,...,n_{k-1},n_{k}+1}},a_{n_1,...,n_k,2} + \frac1{64}r'_k \Big] \subset [a_{{n_1,...,n_{k-1},n_{k}+1}}, a_{{n_1,...,n_k},1}].
$$ 
To simplify the notation at this point we define
${ \mathbf a}_l=a_{{n_1,...,n_k},l}$ and $L_l=[{ \mathbf a}_{l+1},{ \mathbf a}_l]$ so, for example, $ x \in K'_{k+1}=L_{n_{k+1}}$. 
Let $n_0$ be the largest value of $n$  such  that 
$$[x,x+r] \subset \bigcup_{l=n_0}^{n_{k+1}}L_l$$
that is $x+r\in L_{n_0}$.
As $r'_{k+1} \le r$, we have $n_0<n_{k+1}$. Thus $x\le a_{n_1,\ldots,n_k,n_0+1,1}$ and hence $[a_{n_1,\ldots,n_k,n_0+1,1},a_{n_1,\ldots,n_k,n_0+1}] \subset [x,x+r]$.
Therefore
\begin{equation}\label{x+r nagy}
\begin{split}
|[x,x+r]| 
&\ge |[a_{n_1,\ldots,n_k,n_0+1,1},a_{n_1,\ldots,n_k,n_0+1}]| 
= \frac{15}{16} r_{n_1,\ldots,n_k,n_0+1}
= \frac{15}{32} r_{n_1,\ldots,n_k,n_0} \\
&> \frac{15}{64} \sum_{l=n_0}^{n_{k+1}} 2^{n_0-l} r_{n_1,\ldots,n_k,n_0}
= \frac{15}{64} \sum_{l=n_0}^{n_{k+1}} r_{n_1,\ldots,n_k,l}
= \frac{15}{64} \sum_{l=n_0}^{n_{k+1}} |L_l|.
\end{split}
\end{equation}
Consequently 
$$
|E^c \cap [x,x+r]| 
\le \sum_{l=n_0}^{n_{k+1}} |E^c \cap L_l| 
\underset{\text{by Lemma \ref{interval density}}}\le 2\alpha_{k+1} \sum_{l=n_0}^{n_{k+1}} |L_l|
\underset{\text{by \eqref{x+r nagy}}}\le 2\alpha_{k+1} \cdot \frac{64}{15} r.
$$
 and hence 
 \begin{equation}\label{case 2}
|E \cap [x,x+r]| \ge \Big(1-\frac{128}{15}\alpha_{k+1}\Big) r.
 \end{equation}

Combining the estimates \eqref{case 1} and \eqref{case 2} and making use of the fact that $\alpha_k \searrow 0$ and $r'_k \to 0$
as $r\to 0$, we conclude that $E$ is strongly one-sided dense at $x$, as desired.  
\end{proof}

\section{ UDT $\not\Rightarrow $ SUDT, and measurable sets are SUDT modulo sets of measure zero}\label{mainb*}

 The following theorem establishes that the first implication in \eqref{eq:density_implication_chain} cannot be reversed either. 

\begin{theorem}\label{*thmudtnosudt}
There is a UDT set which is not SUDT.
\end{theorem}
\begin{proof}
First we observe that if $F\subset\R$ is measurable, and $x\in I\subset F$ where $I$ is a non-degenerate interval, then $x\in\bigcap_{k=1}^\infty\bigcup_{n=k}^\infty F^{\gamma_n,\delta_n}$ for any sequences $\gamma_n\nearrow 1$ and $\delta_n\searrow 0$.
%for every $(\gamma_n)_{n=1}^\infty,(\delta_n)_{n=1}^\infty$.

We will use the set $E$ defined in the proof of Theorem \ref{thm:main}. This set will be modified 
in a set of measure zero to obtain the set $E'$ which will be a UDT  set, but not a SUDT set.

First we claim that
\begin{equation}\label{nulla}
E_0:=\big\{x\in E :  \text{ }\{x\}\text{ is a  component of }E\big\} \text{ is  a set  of measure $0$.}
\end{equation}

From \eqref{intervallumbeli} it follows that $E_0\subset  \overline{A}$.

According to Lebesgue's density theorem, it is enough to prove that $\overline{A}$ has no density points. 
Take an arbitrary $x\in \overline{A}$.
By the argument following the proof of Lemma \ref{interval density} and ending at \eqref{*capx},
 we can suppose that $x \notin A$  and we can choose a sequence of indices $n_1,n_2, \ldots$ such that  $\{x\}=\bigcap_{k=1}^\infty K_{{n_1,...,n_{k}}}$. 
As  $(a_{{n_1,...,n_{k},1}},a_{{n_1,...,n_{k}}})\cap \overline{A} = \emptyset$  for every $k\in\N$, we have 
$$
|\overline{A}\cap K_{{n_1,...,n_{k}}}|\le \frac{|[a_{{n_1,...,n_{k-1},n_{k}+1}},a_{{n_1,...,n_{k},1}}]|}{r_{{n_1,...,n_{k}}}} = \frac{1}{16}.
$$
Hence $x$ cannot be a density point of $\overline{A}$.
This implies \eqref{nulla}.

Let $\gamma'_n := 1-2^{-n}$ and $\delta'_n := 2^{-100n}$ for every $n\in\N$, and set
$$E':= \bigcap_{k=1}^\infty\bigcup_{n=k}^\infty (E^{\gamma'_n,\delta'_n}\cap E).
$$ 
By the observation at the beginning of this proof $E\setminus E_0\subset E'$.
Hence \eqref{nulla} implies that $E' \subset \bigcap_{k=1}^\infty\bigcup_{n=k}^\infty E'^{\gamma'_n,\delta'_n}$, which means that $E'$ is UDT.

We claim that for every $(\gamma_n)_{n=1}^\infty$ and $(\delta_n)_{n=1}^\infty$ 
with $\gamma_n\nearrow 1$ and $\delta_n\searrow 0$, there is an $x'\in E'\setminus \bigcup_{k=1}^\infty\bigcap_{n=k}^\infty E'^{\gamma_n,\delta_n}$, that is, $E'$ is not SUDT.

Fix $(\gamma_n)_{n=1}^\infty$ and $(\delta_n)_{n=1}^\infty$.

Let $k,n_1,\ldots,n_k\in\N$ and $x\in K_{n_1,\ldots,n_k}\cap E'$. 
As $I_{n_1,\ldots,n_{k-1},n_k}^R\subset[x,x+2r_{n_1,\ldots,n_k}]$ and 
$I_{n_1,\ldots,n_{k-1},n_k+1}^L\subset  [x-2r_{n_1,\ldots,n_k},x]$ , we have
\begin{equation}\label{kicsi}
\begin{gathered}
\max\left\{\frac{|[x-2r_{n_1,\ldots,n_k},x]\cap E|}{2r_{n_1,\ldots,n_k}}, \frac{|[x,x+2r_{n_1,\ldots,n_k}]\cap E|}{2r_{n_1,\ldots,n_k}}\right\} \\
\le \frac{2r_{n_1,\ldots,n_k}-|I_{n_1,\ldots,n_{k-1},n_k+1}^L|} {2r_{n_1,\ldots,n_k}}
= 1-\frac{\alpha_k r_{n_1,\ldots,n_{k-1},n_k+1}}{2r_{n_1,\ldots,n_k}}
= 1-\frac{\alpha_k}{4}.
\end{gathered}
\end{equation}

On the other hand, by \eqref{case 1} and \eqref{case 2},
\begin{equation}\label{nagy}
\begin{gathered}
\text{if $x\in \overline{A}\cap K_{n_1,\ldots,n_k}$ and $r\in \left(0,r_{n_1,\ldots,n_k}\right)$, then } \\
 \max\left\{\frac{|[x-r,x]\cap E|}r, \frac{|[x,x+r]\cap E|}r\right\}\ge 1-384\alpha_k.
\end{gathered}
\end{equation}

We will define sequences $(n'_k)_{k=1}^\infty$, $(k'_j)_{j=1}^\infty$, $(m_j)_{j=1}^\infty$ and $(m'_j)_{j=1}^\infty$ by induction 
 such that the latter three of them are strictly increasing and for every $j\in\N$
\begin{enumerate}[(a)]
\item\label{resze} $K_{n'_1,\ldots,n'_{k'_j}}\cap E \subset {E'}^{\gamma'_{m'_j},\delta'_{m'_j}}$, 
\item\label{diszjunkt} $K_{n'_1,\ldots,n'_{k'_j},n'_{k'_j+1}}\cap E'^{\gamma_{m_j},\delta_{m_j}} = \emptyset$.
\end{enumerate}  
 The existence of such sequences would conclude the proof.  Notably, as $K_{n'_1,\ldots,n'_{k}}\cap E$ is a non-empty closed set for every $k$, and $\lim_{k\to\infty}|K_{n'_1,\ldots,n'_{k}}|=0$, the set 
$$\bigcap_{k=1}^\infty K_{n'_1,\ldots,n'_{k}}\cap E$$
consists of a single point.  If this point is denoted by $x'$, then  by \eqref{resze}, $x' \in \bigcap_{k=1}^\infty\bigcup_{n=k}^\infty (E'^{\gamma'_n,\delta'_n}\cap E) = E'$.
However, $x'\notin \bigcup_{k=1}^\infty\bigcap_{n=k}^\infty E'^{\gamma_n,\delta_n}$ by \eqref{diszjunkt}. %(since $x'\in [a_{n'_1,\ldots,n'_k+1},a_{n'_1,\ldots,n'_k}]$ for every $k\in\N$).
 Consequently, what remains  in order to complete  the proof is the construction of these sequences. 

Set $m'_1:=1$, $k'_1:=10$, and $n'_k:=1$ for every $k\in \N\cap[1,k'_1]$.
Recall \eqref{*calcdef}. 
Since $\delta'_{m'_1}=2^{-100}<2^{-2-9\cdot5}=r_{n'_1,\ldots,n'_{k'_1}}$ and $\gamma'_{m'_1} = 1-2^{-1} < 1-384\cdot10^{-10} = 1-384\alpha_{k'_1}$, 
 by \eqref{nagy} \eqref{resze} is satisfied for $j=1$.

Now, suppose that $j\in\N$ and we have already defined $k'_1,\ldots,k'_j$, $m'_1,\ldots,m'_j$, $m_1,\ldots,m_{j-1}$ and $n'_1,\ldots, n'_{k'_j}$. 

Let $m_j\in\N$ be such that $\gamma_{m_j}>1-\frac{\alpha_{k'_j}}4$, and $m_j>m_{j-1}$ if $j>1$. 
Take a large enough $n'_{k'_j+1}$ to satisfy $2r_{n'_1,\ldots , n'_{k'_j},n'_{k'_j+1}}<\delta_{m_j}$. Then \eqref{diszjunkt} is satisfied by \eqref{kicsi}.

Observe that 
$$
\lim_{i\to\infty} \frac{384\alpha_{k'_{j}+i}}{1-\gamma'_{m'_{j}+i}} 
= \lim_{i\to\infty} \frac{384\cdot10^{-({k'_{j}+i})}}{2^{-(m'_{j}+i)}} 
= 0
$$
and 
$$
\lim_{i\to\infty} \frac{r_{n'_1,\ldots,n'_{k'_j},1_{i}}}{\delta'_{m'_{j}+i}} 
= \lim_{i\to\infty} \frac{r_{n'_1,\ldots,n'_{k'_j}}\cdot2^{-5i}}{2^{-100(m'_{j}+i)}} 
= \infty
$$
where $1_{i}=\underbrace{1,\ldots,1}_{i \text{ times  }}$.
Hence we can select
large enough $i_{j}\in\N$ 
such that $\gamma'_{m'_{j}+i_{j}} < 1-384\alpha_{k'_{j}+i_{j}}$ and $\delta'_{m'_{j}+i_{j}} < r_{n'_1,\ldots,n'_{k'_j},1_{i_{j}}}$. 

Let $m'_{j+1} := m'_j+i_{j}$ and $k'_{j+1} := k'_j+i_{j}$.
Set $n'_k:=1$ for every $k\in[n'_{k'_j+2},n'_{k'_{j+1}}]\cap\N$. Then \eqref{resze} is true by \eqref{nagy}. This concludes the proof.
\end{proof}

 As one can expect, each   measurable set equals an SUDT set  modulo a set of measure zero. We prove this in the next simple theorem. 

 \begin{theorem} \label{thm:sudt_ae}
If $E\subset\R$ is measurable, then there exists $E_{*}\subseteq E$ such that $E_{*}$ has SUDT, and $|E\setminus E_{*}| = 0$. 
That is, every measurable set equals an SUDT set modulo a set of measure zero.
\end{theorem}

 \begin{proof}
Fix a measurable  set  $E\subset\R$.
By the Lebesgue density theorem,  almost every  point of $E$ is a density point of the set, hence we can suppose that $E$ is dense at all of its points.

Now fix $\gamma_n \nearrow 1$.
 For a fixed element $\gamma_n$  note that $E \sse \bigcup_{\delta > 0}E^{\gamma_n, \delta}$  due to our assumption.
 Consequently, for each $n$ we can fix small enough $\delta_n$ such that %$E^{\gamma_n, \delta_n}$ contains $E\cap [-n,n]$ except for a set 
$|E_n| < 1/2^n$, where $E_n =( E \cap [-n,n]) \setminus E^{\gamma_n, \delta_n}$. 
 Now $\sum_{n=1}^{\infty} |E_n|<\infty$ clearly holds. 
 Consequently, according to the Borel--Cantelli lemma, we have that for every $k\in\N$  almost every element of $E\cap [-k,k]$ is contained in  only finitely many $E_n$, and hence by only finitely many $(E \cap [-k,k]) \setminus E^{\gamma_n, \delta_n}$. 
 That is, except for a null set $E'\subseteq E$, any element of $E$ is in $E^{\gamma_n, \delta_n}$ for large enough $n$. 
Now let $E_{*} = E\setminus E'$.   Since  $E^{\gamma_n,\delta_n} = E_{*}^{\gamma_n,\delta_n}$,  it follows  that $E_{*}$ has SUDT.
\end{proof}

 We thank the referee for some comments which improved the exposition and organization of the paper.

\end{document}